\documentclass[11pt]{article}
\usepackage{anysize}\marginsize{3.5cm}{3.5cm}{1.3cm}{2cm}

\makeatletter\@ifundefined{pdfpagewidth}{}{\pdfpagewidth=21.0cm\pdfpageheight=29.7cm}\makeatother 

\usepackage{amssymb}
\makeatletter
\let\orig@item=\@item \def\@item[#1]{\orig@item[\rm #1]}
\makeatother

\newtheorem{satz}{Satz}[section]

\newtheorem{corollary}[satz]{Corollary}
\newtheorem{definition}[satz]{Definition}
\newtheorem{example}[satz]{Example}

\newtheorem{lemma}[satz]{Lemma}
\newtheorem{question}[satz]{Question}
\newtheorem{proposition}[satz]{Proposition}
\newtheorem{remark}[satz]{Remark}

\newtheorem{theorem}[satz]{Theorem}

\newenvironment{proof}[1][Proof]{\trivlist\item[\hskip\labelsep{\it #1.}]}{\hspace*{\fill}$\Box$\endtrivlist}

\renewcommand\emptyset{\varnothing}  
\renewcommand\ge{\geqslant}  
\renewcommand\le{\leqslant}  
\renewcommand\geq{\geqslant}  
\renewcommand\leq{\leqslant}  
\renewcommand\epsilon{\varepsilon}
\renewcommand\phi{\varphi}

\renewcommand\tilde{\widetilde}
\renewcommand\O{{\cal O}}
\renewcommand\P{\mathbb P}
\newcommand\engq[1]{`#1'}

\newcommand\be{\begin{eqnarray*}}
\newcommand\ee{\end{eqnarray*}}
\newcommand\eqnref[1]{(\ref{#1})}
\newcommand\eps{\varepsilon}

\newcommand\inparen[1]{\textnormal{(}{#1}\textnormal{)}}
\newcommand\compact{\itemsep=0cm \parskip=0cm}

\newcommand\mult{{\rm mult}}

\newcommand\newop[2]{\newcommand#1{\mathop{\rm #2}\nolimits}}

\newcommand\calo{{\mathcal O}}
\newcommand\lra{\longrightarrow}
\newop\pr{pr}

\begin{document}

\title{On the Seshadri constants of adjoint line bundles}
\author{Thomas Bauer and Tomasz Szemberg}
\date{November 19, 2010}
\maketitle
\thispagestyle{empty}

\section{Introduction}
   Seshadri constants are interesting invariants of ample line bundles on algebraic
   varieties.
   They were introduced by Demailly in \cite{Dem92}
   and may be thought of as capturing the local positivity of a given line bundle.
   A nice introduction
   to this circle of ideas is given in \cite[Sect.~5]{PAG}, an overview of recent results
   can be found
   in \cite{PSC}. Here we merely recall the basic definition:

\medskip
   Let $X$ be a smooth projective variety, $L$ an ample line bundle on $X$,
   and $x\in X$ a point on $X$. The number
   $$
      \eps(L,x):=\inf_{C\ni x}\frac{L\cdot C}{\mult_xC}
   $$
   is the \emph{Seshadri constant of} $L$ \emph{at} $x$, whereas
   $$
      \eps(L):=\inf_{x\in X}\eps(L,x)
   $$
   is the \emph{Seshadri constant of} $L$.
\medskip

   While $\eps(L)$ is always a positive number,
   Miranda \cite[Example 5.2.1]{PAG} showed that there is no uniform positive lower
   bound for Seshadri constants of ample line bundles on varieties of fixed dimension.
   The purpose of the present note is to show that for
   adjoint line bundles, Seshadri constants exhibit
   surprisingly regular behavior.

   Here is a more detailed description of
   the content of this paper:
   \begin{enumerate}
   \renewcommand\labelenumi{(\theenumi)}
   \item
      While every positive rational number occurs as a Seshadri
      constant of some (integral) ample
      line bundle (Proposition~2.1),
      we show that
      there exists a uniform lower bound
      in the adjoint setting, i.e., for
      ample bundles $K_X+L$, where $L$ is nef
      (Theorem~3.2).
   \item
      On surfaces we show that
      the potential values that $\eps(K_X+L, x)$
      can assume in the interval~$(0,1)$ form a
      monotone increasing sequence with limit~1 (Theorem~4.1).
   \item
      Still on surfaces, we prove that in the \engq{hyper-adjoint}
      setting no values below 1 occur for $\eps(K_X+L,x)$,
      and we classify the borderline case (Theorem~4.6).
   \item
      We complete the picture by looking at the multi-point
      setting, where we provide a uniform lower bound for
      adjoint bundles (Proposition~5.6). On surfaces we answer
      the question corresponding to (2) by showing that there are
      only finitely many possible values (Theorem~5.7).
   \end{enumerate}

   We work throughout over the field of complex numbers.

\paragraph*{Acknowledgement.}
   We would like to thank C.~Ciliberto who sparked our interest
   in studying Seshadri constants in the adjoint setting.
   Also, we thank G.~Heier
   for helpful remarks on an earlier version of our paper.

   The second author was partially supported by a MNiSW grant N N201 388834.

\section{Possible values of Seshadri constants}

   The first observation is that in general every positive rational number
   occurs as a Seshadri constant:

\begin{proposition}\label{ratisesh}
   For every rational number $q>0$ there exists a smooth
   projective surface $X$, an \inparen{integral} ample line bundle $L$ on $X$, and a
   point $x\in X$ such that
   $$
      \eps(L,x)=q \ .
   $$
\end{proposition}

\begin{proof}
   We write $q=\frac{d}{m}$. In the proof we follow closely
   Miranda's idea (cf.~\cite[Prop.~5.2]{Laz93}). We
   construct $X$ as a blow-up of the projective plane, but in
   fact an analogous argument using \cite[Lemma~3.5]{Bau99}
   would work on a suitable blow-up of an arbitrary smooth
   projective surface.
   For a suitably large integer $k$, the following holds true:
   \begin{enumerate}
   \item[(i)]
       There exists an irreducible plane curve $C_1$ of degree $k$ with a point $x$ of multiplicity~$m$.
   \item[(ii)]
      There exists another curve $C_2$ of the same degree $k$ such that
      \begin{itemize}\compact
      \item
         $C_1$ and $C_2$ intersect transversally in $k^2$ distinct points, and
      \item
         all curves in the pencil $V$ generated by $C_1$ and $C_2$ are irreducible.
      \end{itemize}
   \end{enumerate}
   The existence of $C_1$ and $C_2$ is basically a dimension count on sections of $\calo_{\P^2}(k)$ plus
   Bertini's theorem.
   Let now $f:X\lra \P^2$ be the blow-up of $\P^2$ at the intersection points of $C_1$ and $C_2$,
   with exceptional divisors $E_1,\dots,E_{k^2}$. Thus, by (ii), the surface $X$
   is fibred over $\P^1$ by the irreducible
   curves from the pencil $V$. It is easy to verify that the line bundle $L=E+2C$ is ample,
   where $C$ denotes the class of the fiber on $X$ and $E$ is a fixed exceptional divisor.
   With a slight abuse of notation, we denote the preimage of $x$ on $X$ again by $x$.
   Then we have by (i)
   $$
      \frac{L\cdot\tilde{C_1}}{\mult_{x}\tilde{C_1}}=\frac1m
   $$
   for the proper transform $\tilde{C_1}$ of $C_1$,
   so that in any case $\eps(L,x)\leq \frac1m$. But for any irreducible curve $D$ on $X$ passing through
   $x$ (hence different from $E$) and different from
   $\tilde{C_1}$ we have
   \begin{equation}\label{eqn1}
      L\cdot D=(E+2\tilde{C_1})\cdot D\geq 2m \cdot\mult_xD \ ,
   \end{equation}
   so that in fact $\eps(L,x)=\frac1m$.
   Replacing $L$ by $dL$ we get $\eps(dL,x)=\frac{d}{m}$, as
   claimed.
\end{proof}

\begin{remark}\label{ratiseshhdim}\rm
   One can easily generalize this construction to arbitrary dimension $n+2\geq 3$,
   following the idea of \cite[Example 5.2.2]{PAG}:
   to this end, let $Y=X\times\P^n$, where $X$ is the surface constructed in the proof
   of Proposition \ref{ratisesh}, let $M:=\pr_1^*L\times\pr_2^*H$, where $L$
   is the line bundle from the previous proof and $H$ is the hyperplane bundle on $\P^n$.
   Furthermore, let $p\in\P^n$ be a fixed point and
   $Y_p=X\times\left\{p\right\}$. (We identify $Y_p$
   with $X$, in particular we view now the curve $\tilde{C_1}$ as a subvariety of $Y_p$.)
   Then
   $$
      \eps(M, (x,p))=\frac1m \ .
   $$
   In fact, it follows from the projection formula that
   $$
      M\cdot\tilde{C_1}=L\cdot\tilde{C_1}=1\ ,
   $$
   so that in the point $(x,p)$ we have in any case
   $$
      \eps(M,(x,p))\leq\frac1m\ .
   $$
   Let now $D$ be another curve passing through the point $(x,p)$.
   If $D$ is not contained in $Y_p$, then
   $$
      M\cdot D\geq \pr_2^*H\cdot D\geq \mult_{(x,p)}D\ ,
   $$
   which shows that $D$ cannot give a lower Seshadri quotient than $\tilde{C_1}$.
   If on the other hand $D$ is contained in $Y_{p}$, then we conclude the same
   exactly as in \eqnref{eqn1}.
\end{remark}

   Thus we saw that every positive rational number appears as the Seshadri constant
   of some ample line bundle on a variety of dimension $\geq 2$.
   On the other hand it is not known
   -- and it would be extremely interesting to know --
   whether there exist
   irrational Seshadri constants \cite[Remark 5.1.13]{PAG}.

\section{Seshadri constants of adjoint line bundles}
   Now we show that there exists a uniform lower bound on Seshadri constants
   of adjoint line bundles. This is a direct consequence of the following
   result of Angehrn and Siu \cite[Theorem 0.1]{Eff}, but it seems that it
   has not been explicitly noticed so far.

\begin{theorem}[Angehrn-Siu]\label{AngeSiu}
   Let $X$ be a smooth projective variety of dimension $n$ and let $A$ be an ample
   divisor on $X$. Assume that
   $$
      (A^d\cdot Z)\geq \left({{n+1}\choose{2}}+1\right)^d
   $$
   for every irreducible subvariety $Z\subset X$ of positive dimension $d$.
   Then the adjoint line bundle $K_X+A$ is globally generated.
\end{theorem}

\begin{theorem}\label{lowbound}
   Let $X$ be a smooth projective variety of dimension $n$. Let $L$ be a nef line bundle on $X$
   and assume that the adjoint line bundle $K_X+L$ is ample. Then
   $$
      \eps(K_X+L)\geq \frac{2}{n^2+n+4} \ .
   $$
\end{theorem}

\begin{proof}
   We claim that
   $$
      m(K_X+L) \mbox{ is globally generated for } m\geq{{n+1}\choose 2}+2 \ .
   $$
   In fact, take an integer
   $m\geq{{n+1}\choose 2}+2$ and let $A:=(m-1)(K_X+L)+L$.
   This line bundle is ample and it satisfies the inequality
   $$
      (A^d\cdot Z)\geq (m-1)^d
   $$
   for any subvariety $Z\subset X$ of positive dimension $d$.
   Therefore the numerical condition
   in Theorem~\ref{AngeSiu} is satisfied, and hence the adjoint bundle
   $$
      K_X+A=m(K_X+L)
   $$
   is globally generated.

   Now,
   Seshadri constants of globally generated ample line bundles
   are at least $1$
   (see \cite[Example 5.1.18]{PAG}), and this implies the assertion
   after dividing by $m$.
\end{proof}

\begin{remark}\label{heier}\rm
   One can obtain an improved bound for
   $\eps(K_X+L)$ by using Heier's result \cite{Hei}, which says
   that for any nef line bundle $N$ and any integer
   $m\ge (e+\frac12)n^{\frac43}+\frac12 n^{\frac23}+1$
   the
   bundle $K_X+mL+N$ is base-point free.\footnote{Actually, in the
   Main Theorem of \cite{Hei} there is no mention of a nef bundle
   $N$, but as G. Heier informed us, his result remains true in
   the form needed here.}
   Writing $m(K_X+L)=K_X+(m-1)(K_X+L)+L$ and arguing as in the
   proof of Theorem~\ref{AngeSiu}, we get
   $$
      \eps(K_X+L)\ge \frac1{(e+\frac12)n^{\frac43}+\frac12 n^{\frac23}+2}
      \ .
   $$
\end{remark}

\begin{remark}\rm
   It is quite unlikely that the particular bounds on
   $\eps(K_X+L)$ given by
   Theorem~\ref{lowbound} and Remark~\ref{heier} are
   sharp. The important observation is that Seshadri constants of adjoint line bundles
   are bounded from below by a universal number depending only on the dimension
   of the underlying variety.
\end{remark}

   There are two important classes of varieties where all ample line bundles can
   be written as adjoints of ample bundles. On these varieties we have universal lower bounds
   valid for all ample line bundles in all points. In particular we have
   in these cases
   a positive answer
   to the following problem raised by Demailly \cite[Question 6.9]{Dem92}.

\begin{question}\rm
   Let $\eps(X)$ be the infimum of the numbers $\eps(L)$ taken over all
   integral ample line bundles on $X$. Is the number $\eps(X)$ positive, and
   if so, is there an effective lower bound on $\eps(X)$?
\end{question}

\begin{corollary}
   Let $X$ be a variety of dimension $n$ with nef anti-canonical divisor. Then
   $$
      \eps(X)\geq \frac{2}{n^2+n+4} \ .
   $$
   So in particular there is a universal lower bound for Seshadri constants on \inparen{weak} Fano varieties
   and varieties with numerically trivial canonical divisor.
\end{corollary}

\begin{remark}\rm
   Note that the lower bound $\eps(X)\geq \frac{1}{n-2}$ was proved before for
   Fano varieties of dimension $n\geq 3$ by Lee \cite[Theorem 1.1]{Lee} under the additional
   assumption that the anticanonical bundle $-K_X$ be globally generated.
   It seems that the existence of a lower bound valid without any
   restrictions is new.
\end{remark}

\section{Seshadri constants of adjoint line bundles on surfaces}
   For surfaces, i.e., $n=2$, Theorem~\ref{lowbound} gives $\frac15$ as the lower bound.
   One could invoke Reider's theorem in this case to improve this number to $\frac14$.
   However, we show here that the optimal lower bound for the Seshadri constants of an adjoint line bundle on a surface
   is $\frac12$, and we give further restrictions for the possible values
   in the range below $1$.

\begin{theorem}\label{surfaces}
   Let $X$ be a smooth projective surface and $L$ a nef line
   bundle such that $K_X+L$ is ample. If for some point $x\in X$
   the Seshadri constant $\eps(K_X+L,x)$
   lies in the interval $(0,1)$, then
   $$
      \eps(K_X+L,x)=\frac{m-1}m
   $$
   for some integer $m\ge 2$.
\end{theorem}

\begin{proof}
   Let $x$ be a point such that $\eps(K_X+L,x)<1$. Then
   there exists a curve $C\subset X$ such that
   $$
      \eps(K_X+L,x)=\frac{(K_X+L)\cdot C}{\mult_x(C)}=\frac{d}{m} \ .
   $$
   By assumption, we have
   \begin{equation}\label{m geq}
      d\leq m-1 \ .
   \end{equation}
   By the Index Theorem, we have
   \begin{equation}\label{upc2}
      d^2=((K_X+L)\cdot C)^2\geq C^2 (K_X+L)^2 \ ,
   \end{equation}
   so that in any case $C^2\leq d^2$.
   The nefness of $L$ and the adjunction formula imply that we have the following upper bound
   on the arithmetic genus of $C$:
   $$
      p_a(C)=1+\frac12C^2+\frac12C\cdot K_X\leq
     1+\frac12d^2+\frac12C\cdot(K_X+L)=1+\frac{d(d+1)}{2} \ .
   $$
   On the other hand, a curve having a point of multiplicity $m$
   is subject to the following inequality
   \begin{equation}\label{pa-bound}
      p_a(C)\geq {{m}\choose{2}} =\frac{m(m-1)}{2} \ .
   \end{equation}
   Combining these two inequalities, we see that for $m\geq 2$
   we must have
   $$
     d\geq m-1 \ .
   $$
   Together with \eqnref{m geq} this gives the claim.
\end{proof}

   The following lower bound is a direct consequence of the above Proposition.

\begin{corollary}\label{lowsurf}
   Let $X$ be a smooth projective surface and $L$ a nef line
   bundle such that $K_X+L$ is ample. Then
   $$
      \eps(K_X+L,x)\geq\frac12
   $$
   for every point $x\in X$.
\end{corollary}

   Now we show that the bound in Corollary~\ref{lowsurf} is sharp.

\begin{example}\label{gentypeexample}\rm
   Let $X$ be a
   general surface of degree 10
   in weighted projective space $\P(1,1,2,5)$.
   Then $X$ is smooth, $K_X$ is ample with $K_X^2=1$,
   and there is a point $x_0\in X$ such that there exists a
   canonical curve $D\in|K_X|$ with a double point in $x_0$.
   For details we refer to \cite[Example 1.2]{BauSze08}.
   Taking $L$ to be the trivial line bundle, we see
   that
   $$
      \eps(K_X+L,x_0)=\frac{(K_X+L)\cdot D}{\mult_{x_0}D}=\frac12 \ .
   $$
\end{example}
   This example was extreme in the sense that $K_X$ was already ample
   and we took $L$ to be trivial. In the next example
   we show that the Seshadri constant $\frac12$ is possible also at the
   other extreme, i.e., when $K_X$ trivial and $L$ ample.

\begin{example}\label{k3example}\rm
   Let $X$ be a K3 surface with intersection matrix
   $$
      \left(\begin{array}{cc}
      0 & 1\\
      1 & -2
      \end{array}\right) \ .
   $$
   Such a surface exists by \cite[Corollary 2.9]{Mor}.
   Moreover, by \cite[Theorem 2]{Kov} there exist effective curves $\Gamma$ and $E$
   such that $\Gamma^2=-2$, $E^2=0$ generating the Picard group of $X$. In particular,
   we have
   $\Gamma\cdot E=1$. The line bundle $L=\O_X(\Gamma+3E)$ is ample. It intersects
   every curve in the pencil $|E|$ with multiplicity $1$, so that there are
   no reducible curves in the pencil. On the other hand, the elliptic fibration
   defined by $|E|$ must have singular fibers. If $E_0$ is such a singular fiber,
   then it has a double point $x_0$. We have again
   $$
      \eps(K_X+L,x_0)=\frac{L\cdot E_0}{\mult_{x_0}E_0}=\frac12 \ .
   $$
\end{example}

\begin{remark}\rm
   If both $K_X$ and $L$ are ample, then the Seshadri constant of $K_X+L$ is at
   least $1$. To see this, it suffices to repeat the proof of
   Theorem~\ref{surfaces},
   taking into account that the self-intersection of $K_X+L$ is in that case
   at least $4$, so
   that the Index Theorem as in \eqnref{upc2} implies now $C^2\leq\frac14$.
   Combining this again with the lower bound on $p_a(C)$ we get a contradiction to \eqnref{m geq}.
\end{remark}\rm

   One might hope that there exist statements stronger than Corollary~\ref{lowsurf} for
   \engq{hyper-adjoint} bundles, i.e., for adjoints $K_X+L$ of
   \emph{very} ample line bundles $L$. This is indeed the case:

\begin{theorem}
   Let $X$ be a smooth projective surface and $L$ a very ample line bundle
   on $X$ such that $K_X+L$ is ample. Then
   \begin{enumerate}
   \item[a)]
      $\eps(K_X+L)\ge 1$.
   \item[b)]
      If $\eps(K_X+L, x)=1$ for all $x\in X$, then either
      $(X,L)=(\P^2,\O_{\P^2}(4))$
      or $X$ is a ruled surface.
      In the latter case, one has $L=-3C_0+s\cdot f$, where
      $C_0$ is a section, $f$ a fiber of the ruling, and $s$ a
      positive integer.
   \end{enumerate}
\end{theorem}

\begin{proof}
   a)
   Let $x\in X$ and let $C\subset X$ be an irreducible curve
   passing through $x$ with $m:=\mult_xC$.
   We will show that $\frac1m(K_X+L)\cdot C\ge 1$.
   Suppose first that $L\cdot C\le 2$.
   As $L$ is very ample, the curve
   $C$ is then a line or a smooth conic,
   and therefore
   $\frac1m(K_X+L)\cdot C$
   is an integer $\ge 1$ in that case.
   Suppose then that $L\cdot C\ge 3$.
   This inequality implies, when arguing as in the proof of
   Theorem~\ref{surfaces}, that
   \be
      p_a(C)+\frac32\le 1+\frac12C^2+\frac12C(K_X+L)\le
      1+\frac{d(d+1)}2
   \ee
   Using the inequality \eqnref{pa-bound} we get
   $d\ge m$, and this completes the proof of a).

   There is an alternative, adjunction-theoretic approach for the
   proof of assertion~a) as follows.
   The situation described in the proposition was studied by Sommese and Van de Ven:
   In \cite[Theorem 0.1]{SomVdV87} they showed that the adjoint line bundle $K_X+L$
   is globally generated unless
   \begin{itemize}\compact
      \item $X=\P^2$ and $L=\calo(d)$, with $d$ equal to either $1$ or $2$, or
      \item $X$ is a smooth quadric in $\P^3$ and $L$ is the hyperplane bundle, or
      \item $X$ is a $\P^1$ bundle over a smooth curve and $L$ restricted to any fiber is $\calo_{\P^1}(1)$.
   \end{itemize}
   It is easy to see that under our assumptions none of the exceptional cases is possible,
   so that the claim follows using the fact that
   the Seshadri constants
   of ample and globally generated line bundles are $\ge 1$
   (see \cite[Example 5.1.18]{PAG}).

   b) We will make use of the adjunction mapping
   $$
      \phi_{K_X+L}:X\to\P^N \ ,
   $$
   which by the cited result of Sommese and Van de Ven is a
   morphism.

   Suppose first that $(K_X+L)^2=1$. Then the image of
   $\phi_{K_X+L}$ is $\P^2$ and we are done.

   So it remains to consider that case that $(K_X+L)^2\ge 2$.
   By assumption there exists a family of curves $C\subset X$ and
   points $x\in X$ such that
   $$
      \frac{(K_X+L)\cdot C}m=1
   $$
   where $m=\mult_xC$ (cf.~\cite{EL}). We claim first that
   \begin{equation}\label{is-one}
      m=(K_X+L)\cdot C=1 \ .
   \end{equation}
   Indeed, if we had $m\ge 2$, then by \cite[Lemma 1]{Xu}
   (or \cite[Theorem~A]{KSS})
   we would have the inequality
   $$
      C^2\ge m(m-1)+1 \ .
   $$
   Upon using the Index theorem, this implies
   $$
      4(m(m-1)+1)\le 4C^2\le(K_X+L)^2C^2\le ((K_X+L)\cdot
      C)^2=m^2
   $$
   and this is a contradiction, establishing \eqnref{is-one}.

   Next we wish to show that $C^2=0$. In fact, applying the Index
   theorem again, we see that
   $$
      4C^2\le C^2(K_X+L)^2\le((K_X+L)\cdot C)^2=1
   $$
   and hence $C^2\le 0$. The possibility that $C^2<0$ is excluded
   as the curves move in a family.

   We next claim that the curves $C$ are smooth and rational with
   $K_X\cdot C=-2$.
   Indeed, we have
   $K_X\cdot C<(K_X+L)\cdot C=1$, hence $K_X\cdot C\le 0$.
   Using this inequality, together with $C^2=0$ and the
   adjunction formula
   $$
      0\le p_a(C)=1+\frac12 C^2+\frac12 K_X\cdot C
   $$
   implies the claim.

   In order to prove now that $X$ is a ruled surface,
   we show that for some integer
   $k\ge 1$
   the linear series $|kC|$ is a basepoint-free pencil.
   To this end, consider for $k\ge 1$ the short exact sequence
   $$
      0\to\O_X((k-1)C)\to\O_X(kC)\to\O_C(kC)\to 0
   $$
   Its cohomology sequence tells us that if
   $h^0(X,(k-1)C)=h^0(X,kC)$, then $h^1(X,kC)<h^1(X,(k-1)C)$.
   Therefore there exists a $k$ such that
   \begin{equation}\label{h0}
      h^0(X,kC)>h^0(X,(k-1)C)
   \end{equation}
   and hence $|kC|$ is a pencil.
   The curve $C$ is the only possible base curve,
   but we see from \eqnref{h0} that it cannot be the base part of
   $|kC|$.

   Finally,
   after taking the Stein factorization and normalizing, we may
   assume that the general element $f$ of $|kC|$ is irreducible.
   We then see from
   $0\le p_a(f)=1+\frac12k^2C^2+\frac12C\cdot K_X=1-k$
   that $k=1$, and therefore $L\cdot f=3$. This implies that $L$
   is of the form that is
   asserted in the statement of the theorem.
\end{proof}

\begin{remark}\rm
   a) The example of the projective plane $\P^2$ and $L=\calo_{\P^2}(4)$ shows that
   the bound in part a) of the previous proposition cannot be
   improved.

   b) One might hope that in part b) of the theorem it could suffice to ask
   that $\eps(K_X+L,x)=1$ holds for \emph{infinitely many} points
   $x$ instead of requiring it on
   \emph{all} points $x$. But the example of a smooth quartic
   surface $X\subset\P^3$ containing a line $\ell$, with
   $L=\O_X(1)$ and $x\in\ell$ shows that this is not the case.
\end{remark}

\section{Multi-point Seshadri constants of adjoint line bundles}
   Some applications, notably in multivariate interpolation and in Nagata and
   Harbourne-Hirschowitz problems require knowledge of the multi-point version
   of the Seshadri constants defined in the introduction.

\begin{definition}\rm
   Let $X$ be a smooth projective variety and $L$ be an ample line bundle on $X$.
   Let $r$ be a positive integer and $x_1,\dots,x_r$ be arbitrary pairwise distinct
   points on $X$. The real number
   $$
      \eps(L;x_1,\dots,x_r)=\inf\limits_{C\cap\{x_1,\dots,x_r\}\neq\emptyset}\frac{L\cdot C}{\sum_{i=1}^r\mult_{x_i}C}
   $$
   is the \emph{multi-point Seshadri constant} of $L$ at $x_1,\dots,x_r$.
\end{definition}

   It is easy to check that
   \begin{equation}\label{multi-single}
      \eps(L;x_1,\dots,x_r)\geq\frac{1}{\sum_{i=1}^r\frac1{\eps(L,x_i)}} \ ,
   \end{equation}
   so that a lower bound on $\eps(L)$ gives an immediate lower bound
   on $\eps(L;x_1,\dots,x_r)$.

   Without any restrictions on $L$ we can again produce examples of line bundles
   with arbitrary rational multi-point Seshadri constants quite along
   lines of Proposition~\ref{ratisesh}:

\begin{proposition}\label{ratiseshmulti}
   For every rational number $q>0$ and every positive integer $r$
   there exists a smooth
   projective surface $X$, an integral ample line bundle $L$ on $X$, and
   points $x_1,\dots,x_r \in X$ such that
   $$
      \eps(L;x_1,\dots,x_r)=q \ .
   $$
\end{proposition}

\begin{proof}
   It suffices to produce examples with $\eps(L;x_1,\dots,x_r)=\frac{1}{m+r-1}$,
   where $m$ is a given positive integer. All other rational numbers can be obtained
   as multiples of these numbers.

   We modify slightly the construction from the proof of Proposition~\ref{ratisesh}.
   In fact, keeping the notation from this proposition, we simple put $x_1=x$
   and take $x_2,\dots,x_r$ as arbitrary pairwise distinct points on $\tilde{C_1}$.
   Then we have certainly
   $$
      \frac{L\cdot \tilde{C_1}}{\mult_{x_1}\tilde{C_1}+\dots+\mult_{x_r}\tilde{C_1}}=\frac1{m+1+\dots+1}=\frac1{m+r-1} \ .
   $$
   Now, if $D$ is an irreducible curve different from $\tilde{C_1}$, then we have
   \be
      L\cdot D&=&(E+2\tilde{C_1})\cdot D \\
      &\geq& 2(m\cdot\mult_{x_1}D+\mult_{x_2}D+\dots+\mult_{x_r}D) \\
      &\geq& 2\cdot\sum_{i=1}^r\mult_{x_i}D \ ,
   \ee
   and this implies that $\eps(L;x_1,\dots,x_r)$ is computed by $\tilde{C_1}$.
\end{proof}

\begin{remark}\rm
   Of course one can again modify the proof of Proposition~\ref{ratiseshmulti} to
   obtain examples in arbitrary dimension, quite as in Remark \ref{ratiseshhdim}.
\end{remark}

   On the other hand, in the adjoint case, for $X$, $L$ and $K_X+L$ as in Theorem
   \ref{lowbound},
   we see from \eqnref{multi-single} and
   Theorem~\ref{lowbound} that one has
   \begin{equation}\label{multi low1}
      \eps(K_X+L;x_1,\dots,x_r)\geq\frac1r\cdot\frac{2}{n^2+n+4}
   \end{equation}
   for all $r$-tuples $x_1,\dots,x_r\in X$.

   Alternatively one can invoke the following generalization of Theorem \ref{AngeSiu}
   from \cite[Theorem 0.3]{Eff}.

\begin{theorem}[Angehrn-Siu]
   Let $r$ be a positive integer. If
   $$
      (L^d\cdot Z)\geq \left(\frac12n(n+2r-1)+1\right)^d
   $$
   for all irreducible subvarieties $Z\subset X$ of positive dimension $d\geq 1$,
   then
   $$
      K_X+L
   $$
   separates any set of arbitrary $r$ distinct points.
\end{theorem}

   Combining this with the following Lemma leads to the improved lower bound expressed in
   Proposition~\ref{multi low2}.

\begin{lemma}
   Let $r$ be a positive integer and
   let $M$ be a line bundle such that the linear series $|M|$ separates any set of $r+1$ distinct points.
   Then
   $$
      \eps(M;x_1,\dots,x_r)\geq 1
   $$
   for all $r$-tuples $x_1,\dots,x_r$.
\end{lemma}

\begin{proof}
   Let $C$ be a curve passing through at least one of the points $x_1,\dots,x_r$
   and having multiplicities $m_1,\dots,m_r$ at these points. Furthermore let
   $y$ be a point on $C$ distinct from $x_1,\dots,x_r$. Then, by the assumption
   on point separation,
   there exists
   a divisor $D\in|M|$ which contains points $x_1,\dots,x_r$ in its support
   and which avoids $y$. So it intersects $C$ properly, from which we get
   $$
      M\cdot C=D\cdot C\geq\sum_{i=1}^rm_i \ ,
   $$
   and the assertion follows.
\end{proof}

\begin{proposition}\label{multi low2}
   Let $X$, $L$ and $K_X+L$ be as in Theorem \ref{lowbound}. Then
   $$
      \eps(K_X+L;x_1,\dots,x_r)\geq \frac{2}{n^2+(2r+1)n+1} \ .
   $$
\end{proposition}

   This bound is better than \eqnref{multi low1}, but still it is quite unlikely
   that it is sharp. As before we turn now our attention to surfaces,
   where further restrictions are better accessible.

   Corollary \ref{lowsurf} together with \eqnref{multi-single} implies
   that $\eps(K_X+L;x_1,\dots,x_r)\geq\frac1{2r}$. On the other hand
   it is easy to construct examples of  surfaces of arbitrary Kodaira
   dimension, adjoint ample line bundles on them and $r$-tuples
   $x_1,\dots,x_r$ such that $\eps(K_X+L;x_1,\dots,x_r)=\frac1{r}$.
   A sample list of these is the following:
   \begin{itemize}
      \item $\kappa(X)=-\infty$: Take
         $X=\P^2$, $L=\calo_{\P^2}(1)$ and $r$ points on a line,
      \item $\kappa(X)=0$: Take a product $X=E_1\times E_2$ of two elliptic curves, $L=E_1+E_2$ and $r$ points on $E_1$,
      \item $\kappa(X)=1$: Take a product $X=E\times C$ of an elliptic curve $E$ and a smooth curve
         $C$ of genus $\ge 2$, with $L=E+C$ and $r$ points on $E$,
      \item $\kappa(X)=2$: Take the surface $X$ from Example \ref{gentypeexample}, $L=K_X$ and $r$ points on a canonical curve.
   \end{itemize}
   So the interesting question is what values are possible in the range from
   $\frac{1}{2r}$ to $\frac1r$. We show:

\begin{theorem}\label{multi cases}
   We fix an integer $r\geq 2$.
   Let $X$ be a smooth projective surface and let $L$ be a nef line bundle on $X$
   such that $K_X+L$ is ample. If for some distinct points $x_1,\dots,x_r\in X$
   the Seshadri constant $\eps(K_X+L;x_1,\dots,x_r)$ lies in the interval $(0,\frac1r)$, then
   $$
      \eps(K_X+L;x_1,\dots,x_r)=\frac{1}{r+1}
      \quad\mbox{or}\quad \frac1{r+2} \ ,
   $$
   unless $r=2$ and
   $\eps(K_X+L;x_1,x_2)=\frac{2}{5}$.
\end{theorem}

\begin{proof}
   The proof is quite similar to that of Theorem~\ref{surfaces}.
   Let $C$ be a curve on $X$ passing through $x_1,\dots,x_r$ with multiplicities
   $m_1,\dots,m_r$ and such that
   $$
      \frac{(K_X+L)\cdot C}{m_1+\dots+m_r}=\frac{d}{m}<\frac1r \ .
   $$
   Then as in the proof of Theorem~\ref{surfaces} we have
   \begin{equation}\label{eqn2}
      p_a(C)\leq 1+\frac{d(d+1)}{2} \ .
   \end{equation}
   On the other hand, there is the lower bound
   \begin{equation}\label{eqn3}
      p_a(C)\geq{{m_1}\choose2}+\dots+{{m_r}\choose2}=
      \frac12\left(\sum_{i=1}^r m_i^2-\sum_{i=1}^r m_i\right)\geq\frac12\left(\frac1r m^2-m\right).
   \end{equation}
   Using the assumption $rd\leq m-1$ and combining
   \eqnref{eqn2} and \eqnref{eqn3} we get
   $$
      r(m^2-rm-2r)\leq (m-1)\cdot (m+r-1)\ .
   $$
   This implies that either
   \begin{itemize}\compact
      \item[(i)] $m\leq 2r$, or
      \item[(ii)] $r=2$ and $m=5$.
   \end{itemize}
   In case (ii) we get the exceptional value $\frac25$.
   In case (i) we must have $d=1$, and
   using \eqnref{eqn2} we get $p_a(C)\le 2$.
   Therefore there can be at most two double points
   among the~$x_i$, and hence $m$ is bounded by $r+2$.
   This implies the assertion.
\end{proof}

   We conclude by showing that both main cases in the preceding theorem
   actually occur. To obtain $\frac{1}{r+1}$ as a Seshadri constant
   is easy. Indeed, we can either start from Example \ref{gentypeexample}
   or Example \ref{k3example} and take $r-1$ additional smooth points on the curve
   $D$ or $E_0$ respectively. To get $\frac{1}{r+2}$ requires a little
   bit more work.
   The idea is to construct a surface $X$ as in Example
   \ref{gentypeexample}
   containing a canonical curve with two double
   points:

\begin{example}\rm
   In the weighted projective plane $H=\P(1,2,5)$ with
   variables $y,z,w$
   let $C$ be the curve
   that is defined by the homogeneous equation of degree $10$
   $$
      f(y,z,w)=w^2+z^2\cdot (z+y^2)^2\cdot (z-y^2)\ .
   $$
   Note that the curve $C$ omits both singular points $P_1=(0:1:0)$ and $P=(0:0:1)$
   of $H$. It follows that $C$ is irreducible. Indeed, it is elementary to check
   that all polynomials of degree $\leq 5$ vanish either at $P_1$ or $P_2$.
   The curve $C$ has arithmetic genus $2$ and two double points at
   $x_1:=(1:0:0)$ and $x_2:=(1:-1:0)$. We want to realize the curve $C$ as the hyperplane
   section $H\cap X$ of a surface $X\subset\P(1,1,2,5)$ of degree $10$.
   To this end let $D$ be a curve in $H$ defined by a homogeneous
   polynomial $g(y,z,w)$ of degree $9$
   intersecting $C$ transversally. Then we let $X$ be the surface defined
   by the equation
   $$
      F(x,y,z,w):=f(y,z,w)+x\cdot g(y,z,w)=0 \ .
   $$
   We claim that $X$ is smooth. Taking the partial derivative of $F$ with respect to $x$
   we see that the only singular points of $X$ could be the intersection points
   of $C$ and $D$. Since the intersection is transversal, we obtain a local
   coordinate system at each of the intersection points and this shows that $X$
   is smooth.
   For details cf.~\cite[Lemma 2.2]{Bau99}, where an analogous
   construction in $\P^3$ is carried out.

   Now, taking $x_3,\dots,x_r$ on $C$ pairwise different and different from $x_1$
   and $x_2$, we get for the canonical bundle $K_X=\O_X(1)$
   the equation
   $$
      \frac{K_X\cdot C}{\sum_{i=1}^r\mult_{x_i}C}=\frac1{r+2} \ ,
   $$
   as desired.
\end{example}

   We don't know if the exotic value $\frac25$ can be actually obtained as a two-point
   Seshadri constant.

\bigskip
\small
   Tho\-mas Bau\-er,
   Fach\-be\-reich Ma\-the\-ma\-tik und In\-for\-ma\-tik,
   Philipps-Uni\-ver\-si\-t\"at Mar\-burg,
   Hans-Meer\-wein-Stra{\ss}e,
   D-35032~Mar\-burg, Germany.

\nopagebreak
   \textit{E-mail address:} \texttt{tbauer@mathematik.uni-marburg.de}

\bigskip
   Tomasz Szemberg,
   Instytut Matematyki UP,
   PL-30-084 Krak\'ow, Poland.

\nopagebreak
   \textit{E-mail address:} \texttt{tomasz.szemberg@uni-due.de}

\medskip
   \textit{Current address:}
   Tomasz Szemberg,
   Albert-Ludwigs-Universit\"at Freiburg,
   Mathematisches Institut,
   Eckerstra{\ss}e 1,
   D-79104 Freiburg,
   Germany.

\end{document}